\documentclass[12pt]{amsart}

\usepackage{graphicx}
\usepackage{pst-all}
\usepackage{amsmath}
\usepackage{amsfonts}
\usepackage{amssymb}
\usepackage{amsthm}
\usepackage{url}
\usepackage{subfigure}
\usepackage{todonotes}
\input xy
\xyoption{all}
\usepackage{nicefrac}
\usepackage{url}

\newtheorem*{theorem**}{Theorem~A}
\newtheorem*{theorem***}{Theorem~B}
\newtheorem*{theorem****}{Theorem~C}
\newtheorem{theorem}{Theorem}[section]
\newtheorem*{theorem*}{Theorem}
\newtheorem{corollary}[theorem]{Corollary}
\newtheorem*{corollary*}{Corollary}

\newtheorem{proposition}[theorem]{Proposition}
\theoremstyle{definition}
\newtheorem*{stand}{Standing assumption}
\newtheorem{definition}[theorem]{Definition}
\newtheorem{example}[theorem]{Example}
\newtheorem{remark}[theorem]{Remark}

\DeclareMathOperator{\id}{id}
\DeclareMathOperator{\Sh}{Sh}
\DeclareMathOperator{\im}{Im}
\DeclareMathOperator{\rk}{rk}

\thanks{The authors are partially supported by the Spanish Ministerio de Ciencia, Innovaci{\'{o}}n y Universidades (grant PID2021-126124NB-I00).}

\begin{document}

\author{H\'ector Barge}
\address{E.T.S. Ingenieros inform\'{a}ticos. Universidad Polit\'{e}cnica de Madrid. 28660 Madrid (Spain)}
\email{h.barge@upm.es}

\author{Jos\'e M.R. Sanjurjo }
\address{Facultad de Ciencias Matem{\'a}ticas and Instituto de Matem\'atica Interdisciplinar (IMI). Universidad Complutense de Madrid. 28040 Madrid (Spain)}
\email{jose\_sanjurjo@mat.ucm.es}
\title{Shape index, Brouwer degree and Poincar\'e-Hopf theorem}

\begin{abstract} 
In this paper we study the relationship of the Brouwer degree of a
vector field with the dynamics of the induced flow. Analogous relations are
studied for the index of a vector field. We obtain new forms of the Poincar%
\'{e}-Hopf theorem and of the Borsuk and Hirsch antipodal theorems. As an
application, we calculate the  Brouwer degree of the vector field of the
Lorenz equations in isolating blocks of the Lorenz strange set.
\end{abstract}

\maketitle

\section*{Introduction}

The aim of this paper is to study the relationship of the Brouwer degree of
a vector field with the dynamics of the induced flow, in particular with the
dynamical and topological properties of the isolated invariants sets and
their unstable manifolds. Analogous relations are studied for the index of a
vector field, obtaining in this way new forms of the Poincar\'{e}-Hopf
theorem. This classic result has brought a considerable amount of attention in the past ten years and several results in the same spirit have been obtained in different contexts (see for instance \cite{CroGra, KinTro, LiC, MuVan, Rau}).   

As consequences of these relations we also obtain generalizations of Borsuk's and
Hirsch's antipodal theorems for domains that are isolating blocks. We
calculate the Brouwer degree and the index of vector fields in several
situations of dynamical and topological significance. Some applications
include the detection of linking orbits in attractor-repeller decompositions
of isolated invariant compacta and the calculation of the Brouwer degree of
the vector field of the Lorenz equations in isolating blocks of the Lorenz
strange set. Furthermore, we present an
expression of the shape index, originally defined by Robbin and Salamon \cite{RobSal},
in terms of the Euclidean topology. This new expression is quite intuitive
and easy to handle.

We consider flows $\varphi :\mathbb{R}^{n}\times \mathbb{R}\longrightarrow 
\mathbb{R}^{n}$ in the Euclidean space, induced by smooth vector fields $F:%
\mathbb{R}^{n}\longrightarrow \mathbb{R}^{n}$. We will use through the paper
some basic notions of dynamical systems (see \cite{BhSz}) and the Conley index theory (see \cite{Con, Sal}).

We shall make use of the concepts of $\omega$-limit and $\omega^*$-limit of a compactum $X\subset\mathbb{R}^n$ defined as
\[
\omega(X)=\bigcap_{t\geq 0} \overline{X[t,+\infty)}, \quad  \omega^*(X)=\bigcap_{t\leq 0} \overline{X(-\infty,t]}.
\]

We recall that a compact invariant set $K\subset\mathbb{R}^n$ is said to be \emph{isolated} whenever it is the maximal invariant subset of some neighborhod $N$ of itself. A neighborhood $N$ satisfying this requirement is known with the name of \emph{isolating neighborhood}. We shall make extensive use of a special kind of isolating neighborhoods called isolating blocks. An \emph{isolating block} $N$  is an isolating neighborhood with the property that there exist compact subsets $L,L'\subset \partial N$ called the exit and entrance sets such that:
\begin{enumerate}
\item $\partial N=L\cup L'$.
\item For every $x\in L'$ there exists $\varepsilon>0$ such that $x[-\varepsilon,0)\subset\mathbb{R}^n\setminus N$ and for every $x\in L$ there exists $\delta>0$ such that $x[0,\delta)\subset\mathbb{R}^n\setminus N$.
\item For each $x\in L$ there exists $\varepsilon>0$ such that $x[-\varepsilon,0)\subset \mathring{N}$ and for every $x\in\partial N\setminus L$ there exists $\delta>0$ such that $x(0,\delta]\subset\mathring{N}$.
\end{enumerate}

It is well known that an isolating invariant set possesses a basis of neighborhoods comprised of isolating blocks. Moreover, since we are dealing with smooth flows, these blocks can be chosen to be $n$-dimensional manifolds with boundary, satisfying that $L$ and $L'$ are $(n-1)$-dimensional submanifolds of $\partial N$ of with $\partial L=L\cap L'=\partial L'$ (see \cite{ConEast} or \cite[Ap\'endice~A.2]{SGth}). In this case, the points of $L\cap L'$ are exactly those where the vector field is tangent to $\partial N$. All the isolating blocks considered in this paper will be assumed to be of this kind without explicitly mention it.

 Among isolated invariant sets attractors play a central role. An \emph{attractor} $K$ is a compact invariant set that is stable and possesses a neighborhood $U$ such that the omega-limit $\omega(x)$ of each point in $U$ is non-empty and is contained in $K$. The condition on stability means, roughly speaking, that the positive semi-trajectories of nearby points remain nearby. More precisely, if $V$ is any neighborhood of $K$, there exists a neighborhood $U$ of $K$ with the property that $U[0,+\infty)\subset V$. An attractor is said to be \emph{global} if the neighborhood $U$ can be chosen to be the total space. Repellers are defined in an analogous way using the negative omega-limit $\omega^*$ and negative semi-trajectories. Notice that a repeller is just an attractor for the flow obtained by changing the sign of the time variable.  
 
We are going to use some notions from algebraic topology, including duality theorems. All the material we are going to need is covered in the books by Hatcher \cite{Hat}, Munkres \cite{Munk} and Spanier \cite{Span}. We use the notation $H_*$ and $H^*$ to denote the singular homology and cohomology functors and $\check{H}^*$ for the \v Cech cohomology functor, all of them with integer coefficients unless otherwise specified. We recall that \v Cech and singular cohomology theories coincide on manifolds and, more generally, polyhedra and pairs of such spaces. We say that a pair of spaces $(X,A)$ is of \emph{finite type} whenever $\check{H}^k(X,A)$ is finitely generated for every $k$ and non-zero for only a finite number of values of $k$. Pairs of compact manifolds or, more generally, polyhedra are examples of pairs of finite type. While attractors and repellers of flows defined on euclidean spaces are  of finite type (\cite[Corollary~4.2]{KapRod}), this is not the case for every isolated invariant set. For instance,  \cite[Remark~9]{BNLA} shows an example of a flow on $\mathbb{R}^3$ which has the hawaaian earring as an isolated invariant set. In order to make our statements cleaner we shall make the following standing assumption.

\begin{stand}
Whenever $K$ is an isolated invariant set, we shall assume, without further mention, that $K$ is of finite type. 
\end{stand}

If a pair $(X,A)$ is of finite type, its \emph{Euler characteristic} is defined as follows
\[
\chi(X,A)=\sum_{k} (-1)^k\rk \check{H}^k(X,A). 
\]

We shall make use of the following property of the Euler characteristic \cite[Exercise~B.1, p. 205]{Span}: If two of the three $(X,A)$, $X$, $A$, are of finite type, then so is the third and
\[
\chi(X)=\chi(X,A)+\chi(A).
\]

Notice that if $(X,A)$ is a pair of manifolds or polyhedra, then the Euler characteristic defined in an analogous way using singular cohomology coincides with the previous one.

The \emph{Conley index} of an isolated invariant set $K$, denoted $h(K)$, is defined as the pointed homotopy of the quotient space $(N/L,[L])$ where $N$ is any isolating block for $K$. Notice that, while different isolating blocks may represent different homotopy types, the Conley index only depends on $K$. The \emph{cohomology index} $CH^*(K)$ is defined to be the \v Cech cohomology of the pointed space $(N/L,[L])$. Using the strong excision property of \v Cech cohomology we get that the cohomology index is isomorphic to $\check{H}^*(N,L)$. 

Given an isolating block $N$ for an isolated invariant set $K$, we shall denote by $\deg(F,N)$ to the degree of $F_{|_{\mathring{N}}}$ and, if $K$ has only a finite number of singularities, by $I(F_{|_N})$ to the total index of $F_{|_N}$, that is, the sum of the indices of the singularities. When we refer to $I(F_{|_N})$, we
say that the index is defined if $F_{|_N}$ has a finite number of singular points.  For a detailed treatment of mapping degree theory, including the index theory of vector fields, see the book by Outerelo and Ruiz \cite{OuRu}.

We shall make use of the following result, obtained by  Srzednicki \cite{SrFM}, McCord \cite{McH}, Fotouhi and Razvan \cite{RaFo}, in different levels of generality, that relates degree of a vector field near an isolated invariant set, with its Conley index. We present here a slightly different but equivalent version of the one presented by Izydorek and Styborski in \cite[Theorem~4.2]{IzSty}.

\begin{theorem}\label{thm:SrMc}
Let $\varphi:\mathbb{R}^n\times\mathbb{R}\longrightarrow\mathbb{R}^n$ be a flow induced by a smooth vector field $F$ defined on $\mathbb{R}^n$. Suppose that $K$ is an isolated invariant set for $\varphi$ and $N$ an isolating block for $K$. Then,
\[
\deg(F,N)=(-1)^n\chi(h(K)).
\] 
Moreover, if $K$ contains only a finite number of equilibria then
\[
I(F_{|_N})=(-1)^n\chi(h(K)).
\]
\end{theorem} 

Notice that the Euler characteristic of the Conley index is well-defined, taking into account that the pair $(N,L)$ used to compute it can be chosen to be a pair of compact manifolds. The second part of the statement follows from the fact that the Brouwer degree is,
by the additivity property, the sum of the indices of all the singular points of $F$ in $N$. 
 
Finally, we will also use some elementary facts from Borsuk's homotopy theory (named Shape Theory by him). This theory was introduced by K. Bosuk in 1968 in order to study homotopy properties of compacta with bad local behaviour for which the classical homotopy theory is not well suited. We are not going to make an extensive use of Borsuk's homotopy theory, in particular we are only interested in the following very simple situation:
Consider a compact metric space $K$, a closed subspace $K_{0}$ and a
sequence of maps $f_{k}:K\longrightarrow K$ such that ${f_{k}}_{|_{K_{0}}}:K_{0}%
\longrightarrow K_{0}$ (i.e. ${f_{k}}_{|_{K_{0}}}$ maps $K_{0}$ to itself) and the
following conditions hold for almost every $k$:
\begin{enumerate}
\item For every neighborhood $U$ of $K_{0}$ in $K$
we have $f_{k}\simeq f_{k+1}$ in $U$.

\item  $f_{k}\simeq \id_{K}.$

\item ${f_{k}}_{|_{K_{0}}}\simeq \id_{K_{0}}$ in $K_{0}.$
\end{enumerate}

Then $K$ and $K_0$ have the same shape (there is an analogous statement for pointed
shape). We shall use the notation $\Sh(K)=\Sh(K_{0})$ to denote that both $K$ and $K_0$ have the same shape. We shall also make use of the following fact from shape theory
\begin{enumerate}
\item If $X$ and $Y$ have the same homotopy type, then they have the same shape.
\item If $X$ and $Y$ are polyhedra (or more generally, ANR), $X$ and $Y$ have the same shape if and only if they have the same homotopy type.
\item If $X$ and $Y$ have the same shape, then they have isomorphic \v Cech cohomology groups. 
\end{enumerate}

More information about the theory of shape can be found in the books by Borsuk \cite{Bormono}, by Marde\v si\'c and Segal \cite{MarSe} and by Dydak and Segal \cite{DySe}. Some applications of shape theory to dynamics can be seen in \cite{KapRod, SanMul}.

\section{Shape index, initial sections and the degree of a vector field}

In \cite{BSJMA}, the authors proved that some parts of the unstable manifold of an
isolated invariant set admit sections that carry a considerable amount of
information. These sections enable the construction of parallelizable
structures which facilitate the study of the flow. 

\begin{definition}
Let $K$ be an isolated invariant compactum and let $S$ be a compact section
of the truncated unstable manifold $W^{u}(K)\setminus K$. Then $S$ is said to be an
\emph{initial section} provided that $\omega^{\ast }(S)\subset K.$
\end{definition}

If $S$ is an initial section we define
\[
I_{S}^{u}(K)=S(-\infty ,0].
\]
Obviously, $I_{S}^{u}(K)=\{x\in W^{u}(K)\setminus K\mid xt\in
S$ with $t\geq 0\}$. In accordance with this terminology we say that $%
I_{S}^{u}(K)\cup K$ is an \textit{initial part of the unstable manifold }of $%
K$ and we denote it by $W_{S}^{u}(K).$ In \cite{BSJMA} it was proved that, although $%
I_{S}^{u}(K)$ depends on $S$, all the initial parts have basically the same
structure. More specifically, if $S$ and $T$ are initial sections of $W^u(K)$, the pairs $(W^u_S,S)$ and $(W^u_T,T)$ are homeomorphic.

Analogous notions known as \emph{final section} and \emph{final part} of the stable manifold can be defined and have similar properties. 

If $N$ is an isolating block for $K$, we denote by $N^{-}$ the negative asymptotic set, that is, the set $\{x\in N\mid xt\in N$ for
every $t\leq 0\}$. Set $n^{-}=N^{-}\cap L$. It is easy to see that $N^{-}$ is an initial part of the unstable manifold with initial section $n^{-}$. The positive asymptotic set $N^+$ is defined in an analogous way and is a final part of the stable stable manifold with final section $n^+=N^+\cap L'$.

In this paper we make some use of the shape index of an invariant isolated
set $K$. The shape index $S(K)$ was introduced by Robbin and Salamon in \cite{RobSal}
as $\Sh(N/L,\ast )$, where $\ast =[L].$ The cohomology of the shape index is
the classical cohomological (Conley) index. In \cite{Sanjnon}, the second author showed
that the shape index can be represented in terms of compact sections of the
unstable manifold endowed with the intrinsic topology (as defined also by
Robbin and Salamon). The constraint of the intrinsic topology is
substantial, since this topology is not very intuitive and difficult to
handle, so we believe that an expression of the shape index in terms of the
Euclidean (or extrinsic) topology is much more useful and we find it in the
first result of the paper. A crucial element of this expression is the use 
of the initial sections of truncated manifolds as defined above.

\begin{theorem}\label{thm:1}
Let $\varphi :\mathbb{R}^{n}\times \mathbb{R}\longrightarrow \mathbb{R}^{n}$ be
a flow (not necessarily differentiable) and $K$ an isolated invariant set of 
$\varphi $. Let $W^{u}(K)$ the unstable manifold of $K$, $S$ an initial
section of $W^{u}$ and $W_{S}^{u}(K)$ the corresponding \emph{initial part
of the unstable manifold }of $W^{u}(K)$ .Then the shape index $S(K)$ is $%
\Sh(W_{S}^{u}(K)/S,\ast )$, that is, the pointed shape of the quotient set $%
W_{S}^{u}/S$ where $\ast =[S]$. Furthermore, if $CS$ is the cone over $S$
then $S(K)=Sh(W_{S}^{u}\cup CS,\ast ),$ were $\ast $ is the vertex of the
cone.
\end{theorem}

\begin{proof}
Let $N$ be an isolating block of $K$ and $L$ its exit set. Since all the pairs $%
(W_{S}^{u},S),$ where $S$ is an inital section, are homeomorphic, we can
limit ourselves to the pair $(N^{-},n^{-})$. Let $\alpha :N\setminus N^{+}\longrightarrow 
\mathbb{R}$ the map defined by
\[
\alpha (x)=\max \{t\in \mathbb{R}%
\mid x[0,t]\subset N\}.
\]
Then for every $k\in \mathbb{N\cup \{}0\}$ we define
the map $f_{k}:N\longrightarrow N$ by
\[
f_{k}(x)=
\begin{cases}
kx & \mbox{if}\quad x[0,k]\subset N \\
\alpha (x)x & \mbox{otherwise.}
\end{cases}
\]

The map $f_{k}$ is continuous and fixes all points in $L$ (this is
essentialy Wazewski's Lemma \cite[Theorem~2]{Wazewski}). Suppose that $U$ is an open neighborhood of 
$N^{-}\cup L$ in $N$. We claim that there is a $k_{0}\in \mathbb{N}$ such
that $\im f_{k}\subset U$ and $f_{k}\simeq f_{k+1}$ in $U$ for every $k\geq
k_{0}$, and the homotopy leaves al points in $L$ fixed. In order to prove it
we show that there is a positive $s_{0}$ such that $N^{+}[s_{0},\infty
)\subset U$. Otherwise, there would be sequences $x_{n}\in N^{+}$ and $%
s_{n}\longrightarrow \infty $ such that $x_{n}s_{n}\rightarrow y\in N^{+}-U$.
Then, $\gamma ^{-}(y)\subset N$ and, since $y\in N^{+}$, the whole
trajectory $\gamma (y)$ would be contained in $N\setminus K$ in contradiction with
the assumption that $N$ is an isolating block of $K$. Furthermore, there is
a $s_{1}\geq s_{0}$ with the property that $xt\in U$ for every $x\in N\setminus N^{+}$
and for every $t$ such that $s_{1}\leq t\leq \alpha(x)$. Otherwise there
would be sequences $x_{n}\in N$, $t_{n}\rightarrow \infty $ with $%
x_{n}t_{n}\notin U$, $x_{n}t_{n}\rightarrow y\in N$ and $x_n[0,t_{n}]%
\subset N$. But this would imply that $y\in N^{-}$, in contradiction with
the fact that $y\notin U$, as limit of $x_{n}t_{n}$. We obtain from this
that $\im f_{k}\subset U$ for $k\geq s_{1}$ and select an index $k_{0}\geq
s_{1}.$ It is clear that the homotopy $h_{k}:N\times \lbrack 0,1]\rightarrow
N $ defined by

\[
h_{k}(x,t)= 
\begin{cases}
f_{k}(x)t& \mbox{if}\quad  f_{k}(x)[0,t]\subset N \\
x\alpha (x)& \mbox{otherwise}
\end{cases}
\]
links $f_{k}$ and $f_{k+1}$ in $U$ leaving all points in $L$ fixed.
Furthermore, the map

\[
h(x,t)=
\begin{cases}
xtk & \mbox{if}\quad x[0,tk]\subset N \\
x\alpha(x) & \mbox{otherwise,}
\end{cases}
\]

defines a homotopy $h:N\times \lbrack 0,1]\longrightarrow N$ linking $\id_{N}$
with $f_{k}.$ Similarly, it can be seen that $f_{{k}|_{N^{-}\cup L}}$ is a map $%
N^{-}\cup L\longrightarrow N^{-}\cup L$ homotopic to $\id_{N^{-}\cup L}$, with the
homotopy fixing all points in $L$.

Notice that the subspace $(N^{-}\cup L)/L\subset N/L$ can be identified with 
$N^{-}/n^{-}$. We use the notation $\ast $ to designate both the point $%
[L]\in N/L$ and the point $[n^{-}]\in N^{-}/n^{-}$. Now consider the
composition $\bar{f}_{k}=p\circ f_{k}:N\longrightarrow N/L$, where $%
p:N\longrightarrow N/L$ is the natural projection. This map induces a map $\hat{f%
}_{k}:N/L\longrightarrow N/L$ such that $\hat{f}_{k}=p\circ \bar{f}_{k}$. We
then have a sequence of maps $\hat{f}_{k}:N/L\longrightarrow N/L$ such that $%
\hat{f}_{k|_{N^{-}/n^{-}}}:N^{-}/n^{-}\longrightarrow N^{-}/n^{-}$ and the
following conditions are satified for almost every $k$:

\begin{enumerate}

\item For every neighborhood $\hat{U}$ of $N^{-}/n^{-}$ in $N/L$ we have $\hat{f%
}_{k}\simeq \hat{f}_{k+1}$ in $\hat{U}$.

\item $\hat{f}_{k}\simeq \id_{N/L}.$

\item $\hat{f}_{{k}|_{N^{-}/n^{-}}}\simeq \id_{N^{-}/n^{-}}$.

\item  All the homotopies leave the point $\ast $ fixed.
\end{enumerate}

It follows from this that $\Sh(N/L,\ast )=\Sh(N^{-}/n^{-},\ast )$ and,
therefore, $\Sh(W_{S}^{u}/S,\ast )=S(K)$ (where $\ast =[S].$)

To finish the proof it remains to observe that by \cite[Corollary~3, pg. 247]{MarSe}, the natural projection $W_{S}^{u}\cup CS\longrightarrow W_{S}^{u}/S$ is a
pointed shape equivalence and, therefore $S(K)=\Sh(W_{S}^{u}\cup
CS,\ast ),$ were $\ast $ is the vertex of the cone.
\end{proof}

\begin{remark} 
The statement in Theorem~\ref{thm:1} does not hold if the section $S$ is not initial. An example of
a compact section $S$ is given in \cite[Fig~2, pg. 840]{BSJMA} which is not initial and
such that $Sh(W_{S}^{u}(K)/S,\ast )$ is not the shape index $S(K)$.
\end{remark}

By combining Theorem~\ref{thm:1} together with Theorem~\ref{thm:SrMc} we obtain the following corollary that relates the degree of a vector field (or the index when defined) near an isolated invariant set, its Euler characteristic and the Euler characteristic of an initial section.

\begin{corollary}\label{cor:genPoinH}
Let $\varphi :\mathbb{R}^{n}\times \mathbb{R}\longrightarrow \mathbb{R}^{n}$ be
a flow induced by the smooth vector field $F:\mathbb{R}^{n}\longrightarrow 
\mathbb{R}^{n}$and $K$ an isolated invariant set of $\varphi $. Let $%
W^{u}(K) $ the unstable manifold of $K$, $S$ an initial section of $W^{u}$
and $W_{S}^{u}(K)$ the corresponding i\textit{nitial part of the unstable
manifold}. If $N$ is an isolating block of $K$ then 
\begin{equation}
\deg (F,N)=(-1)^{n}(\chi (K)-\chi (S)).
\label{eq:1}
\end{equation}
Moreover, if the index $I(F_{|_N})$ is defined then 
\begin{equation}
I(F_{|_N})=(-1)^{n}(\chi (K)-\chi (S)).
\label{eq:2}
\end{equation}
\end{corollary}

\begin{proof}
First observe that, if $L$ is the exit set of $N$ we have that $\chi(N,L)$ is defined and 
\[
\chi(S(K))=\chi(N,L)=\chi(h(K)).
\]
As before, we can assume that $W_{S}^{u}(K)=N^{-}$ and $S=n^{-}.$ By an argument similar to that used
in the proof of the Theorem~\ref{thm:1}, it is easy to see that $\Sh(N^{-})=\Sh(K)$. Then $%
\check{H}^{r}(N^{-})=\check{H}^{r}(K)$ for every $r$ and, thus, $\chi
(N^{-})=\chi (K).$ Moreover, Theorem~\ref{thm:1} ensures that $S(K)=\Sh(N^{-}/n^{-},\ast )$. Hence, as a consequence of the strong excision property of \v Cech cohomology we obtain that
\[
\chi(N,L)=\chi(N^-,n^-)=\chi(N^-)-\chi(n^-)=\chi(K)-\chi(S).
\] 
Notice that, since $\chi(N^-,n^-)$ and $\chi(N^-)$ are defined so is $\chi(n^-)$ (hence, $\chi(S)$).

The equalities \eqref{eq:1} and \eqref{eq:2} follow from Theorem~\ref{thm:SrMc}.
\end{proof}

This corollary allows in many cases to establish a direct relation of the Brouwer degree and the total index with the topology of the invariant set $K$, as we show in various results of the paper. In other cases we also need some knowledge of the initial section of the unstable manifold, which is often easy to compute. This gives an alternative method
to the one provided by Theorem~\ref{thm:SrMc} that has some advantages in certain cases. 

A direct consequence of Corollary~\ref{cor:genPoinH} is a particular case of a result obtained by Sredznicki \cite[Lemma~6.2]{SrFM}.

\begin{corollary}
If $K$ is a continuum in $\mathbb{R}^{2}$ then $\deg (F,N)\leq
\chi (K)$. Also $I(F_{|_N})\leq \chi (K)$ when $I(F_{|_N})$ is defined.
\end{corollary}

\begin{proof}
In \cite{BSJMA} it was shown that for $n=2$ the section $S$ is a disjoint finite union
of circles and (possibly degenerate) topological intervals and, consequently, $\chi (S)\geq
0.$ Then $\deg (F,N)=\chi (K)-\chi (S)\leq \chi (K)$.
\end{proof}

The following corollary deals with the important particular case of Theorem~\ref{cor:genPoinH} when the initial part of the unstable manifold \textbf{is a genuine manifold} whose
boundary is the initial section $S$.

\begin{proposition}\label{prop:unman}
Suppose that $W_{S}^{u}$ is an $m$-dimensional manifold with boundary $%
\partial W_{S}^{u}=S$. Then
\[
\deg (F,N)=(-1)^{(n+m)}\chi(K).
\]
In particular, $\deg (F,N)$ agrees with $\chi (K)$ if the
parities of $n$ and $m$ are the same and with $-\chi (K)$ otherwise.
The same statement is valid for $I(F_{|_N})$ when  defined.
\end{proposition}

\begin{proof}
Since by Corollary~\ref{cor:genPoinH} 
\[
\deg(F,N)=(-1)^n(\chi(K)-\chi(S)),
\]
we only have to compute $\chi(S)$. Taking into account that $W^u_S(K)$ is a genuine $m$-manifold whose boundary is $S$, it follows that $S$ is a closed $(m-1)$-manifold. If $m$ is even, then $m-1$ is odd and Poincar\'e duality ensures that $\chi(S)=0.$ On the other hand, if $m$ is odd, Lefschetz duality, together with the fact that the $\Sh(W^u_S)=\Sh(K)$ ensure that 
\[
\chi(W^u_S,S)=-\chi(W^u_S)=-\chi(K).
\]  
As a consequence, $\chi(S)=2\chi(K)$ and the result follows.
\end{proof}

The classical Poincar\'{e}-Hopf Theorem is the most important particular
case of Proposition~\ref{prop:unman}:

\begin{corollary}\label{coro:outcar}
If $F$ points outward in $\partial N$ and $I(F_{|_N})$ is defined then $I(F_{|_N})=\chi (N)$. The same holds for $\deg(F,N)$. In this case there is no requirement for $I(F_{|_N})$ to be defined.
\end{corollary}

\begin{proof}
Since $F$ points outwards in $\partial N$ it follows that the maximal invariant set contained in $N$ must be a repeller. Moreover, $N$ is a negatively invariant neighorbood of $K$ and, as a consequence, we can take $%
W_{S}^{u}=N$, $S=\partial N,$ and $n=m$. Since $\check{H}^*(
W_{S}^{u})=\check{H}^*(K)$ the result follows from Proposition~\ref{prop:unman}.  
\end{proof}

The following result refers to flows that have a global repeller.

\begin{corollary}
If $\varphi$ has a global repeller $K$ then $\deg (F,N)=1$ for every
isolating block containing $K$ and $I(F_{|_N})=1$ when defined.
\end{corollary}

\begin{proof}
If $K$ is the global repeller of $\varphi$ then $K$ has the \v Cech cohomology groups of a point by \cite[Theorem~3.6]{KapRod}. Since the degree does not depend on the choice of the isolating block, we may assume that $N$ is  negatively invariant, i.e, such that $N=N^-$. Hence, $\chi(N)=\chi(K)=1$ and the result follows from Corollary~\ref{coro:outcar}.
\end{proof}

In dimension 2 we obtain a kind of reciprocal to the Poincar\'{e}-Hopf
theorem and a nice characterization of the flows such that $I(F_{|_N})=\chi (N)$.

\begin{proposition}\label{prop:tangent}
If $K$ is a continuum in $\mathbb{R}^{2}$ and $I(F_{|_N})$ is defined then the vector field $F$ is tangent to $\partial N$ in exactly  $2(\chi (N)-I(F_{|_N}))$ points. As a consequence, $I(F_{|_N})=\chi(N)$ if and only
if $F$ either points outward or inward in every component of $%
\partial N$. 
\end{proposition}

\begin{proof}
The first part of the statement follows from the fact that the points of tangency are exactly the points of $L\cap L'=\partial L$, where $L'$ is the entrance set. Since $L$ is a compact $1$-dimensional manifold, it is a disjoint union of circles and closed intervals. Hence, $\partial L$ consists of the endpoints of each interval component of $L$. Since each interval has exactly two endpoints and $\chi(L)$ is just a count of the number of interval components of $L$,  the result follows from Theorem~\ref{thm:SrMc}. The second part of the statement is a direct consequence of this discussion.
\end{proof}

\begin{example}

Let $\varphi$ be a flow in $\mathbb{R}^2$ induced by a vector field $F$ and suppose that $K$ is an isolated periodic trajectory. Then, it is not difficult to see that $K$ admits an isolating block $N$ that is a closed annulus with two different boundary components, each contained in a different component of $\mathbb{R}^2\setminus K$. Since $K$ does not contain fixed points, then
\[
I(F_{|_N})=0=\chi(N),
\]
and, Proposition~\ref{prop:tangent} ensures that the vector field points either outward or inward in every component of $\partial N$. Hence, we have three mutually exclusive possibilities:
\begin{enumerate}
\item $F$ points inward in both components of $\partial N$ and, hence, $K$ is an attractor.
\item $F$ points outward in both components of $\partial N$ and, hence, $K$ is a repeller.
\item $F$ points inward in one component of $\partial N$ and outward in the other. In this case $K$ is neither an attractor nor a repeller. 
\end{enumerate}
Although these three possibilities cannot be distinguished only using the index, they can be distinguished by using the Conley index. Indeed, in the first case, $L=\emptyset$ and, hence, the effect of collapsing is equivalent to make the disjoint union of $N$ with a point $\{*\}$ not belongin to $N$. Since $N$ is an annulus, it has de homotopy type of the circle $S^1$ and, hence, the Conley index of $K$  is the pointed homotopy type of $(S^1\cup\{*\},*)$ where $*$ is a point that does not belong to $S^1$. 

In the second case $L=\partial N$ and $(N/L,[L])$ is a pinched torus that is pointed homotopy equivalent to the wedge $(S^2\vee S^1,*)$. 

Finally, in the third case, $L$ is just one component of $\partial N$. Since $N$ is homeomorphic to the product $S^1\times[0,1]$, $(N/L,[L])$ is just the cone $(CS^1,*)$ that is contractible. It follows that the Conley index of $K$ is trivial.

This shows that the Conley index is a finer invariant than the index of a vector field.  
\end{example}

\section{Brouwer degree of vector fields near non-saddle sets}

In this section we study the Brouwer degree of a vector field in a vicinity of a special class of isolated invariant sets called non-saddle. 

We start by recalling that an invariant set $K$ is said to be \emph{non-saddle} if it satisfies that for every neighborhood $U$
of $K$ there exists a neighborhood $V$ of $K$ such that for all $x\in V$ either $x[0,+\infty)\subset U$ or $x(-\infty,0]\subset U$. Otherwise $K$ is said to be \emph{saddle}. We shall only consider non-saddle sets that are also isolated. Attractors, repellers and unstable attractors with mild forms of instability are some examples of non-saddle sets. Isolated non-saddle sets are characterized by possesing arbitrarily small isolating blocks of the form $N=N^+\cup N^-$ (see \cite[Proposition~3]{BSdis}). Moreover, if $K$ is connected, every isolating block is, in fact, of this form. Notice that if $N$ is an isolating block of this form, the vector field points either inward or outward in each connected component of $\partial N$. Using the homotopies provided by the flow, it easily follows that $\check{H}^*(K)\cong H^*(N)$ and, therefore, $K$ is of finite type. Another property that we shall use in the sequel is that the union of the components of the boundary of an isolating block of the form $N=N^+\cup N^-$ in which the vector field points outward is an initial section of the unstable manifold of $K$. In an analogous way, the union of thoose components of $\partial N$ in which the flow points inward is a final section of the stable manifold. Hence, for isolated non-saddle sets of smooth flows on $\mathbb{R}^n$ initial and final sections of the unstable and stable manifolds are closed manifolds of dimension $n-1$. For more information on isolated non-saddle sets the reader can see \cite{Bdcds, BSbif, BSdis}.

In view of this, the following result is a far-reaching
generalization of the Poincar\'{e}-Hopf Theorem. Furthermore, it provides a
nice characterization of non-saddle sets for flows in the plane.

\begin{proposition}\label{prop:ns1}
 Suppose that $K$ is a non-saddle continuum, $N$ is a connected isolating block
of $K$ and $S$ and $S^*$ an initial and a final section of the truncated unstable and stable manifolds of $K$ respectively. Suppose also that $I(F_{|_N})$ is defined. Then,
\begin{enumerate}
\item  If the dimension $n$ is even
then $I(F_{|_N})=\chi (K)=\chi (N)$.
\item If $n$ is odd then $I(F_{|_N})=\frac{1}{2}%
(\chi (S^{\ast })-\chi (S))=-\chi (N)+\chi (S^{\ast })$ (note that if $K$ is
a repeller then $\chi (S^{\ast })=0$ and, therefore, $I(F_{|_N})=-\chi (N)$).
\end{enumerate}
Moreover, if $n=2$ and $K$ is an arbitrary isolated invariant continuum then $I(F_{|_N})=\chi (K)$ if and only if $K$ is non-saddle.

An analogous statement is valid for $\deg (F,N)$. In this case there is no
requirement that $I(F_{|_N})$ be defined.
\end{proposition}

\begin{proof}
We may assume without loss of generality that $S=n^-$ and $S^*=n^+$. If $n$ is even, since $S$ is a closed manifold of odd dimension it follows that $\chi (S)=0$ and, hence, $I(F_{|_N})=\chi (K)=\chi (N).$ 

Suppose now that $n$ is odd. Taking into account that $\partial N=S\cup S^*$, Lefschetz duality applied to the pair $(N,\partial N)$ yields
\[
H^{\ast }(N,S\cup S^{\ast })=H_{n-\ast }(N), 
\]%
where the homology and the cohomology are taken in dual dimensions relative
to $n.$

Since $n$ is odd we deduce from the former expression that 

\[
\chi (N,S\cup S^{\ast })=-\chi (N).
\]
On the other hand, 
\[
\chi (N,S\cup S^{\ast })=\chi (N)-\chi (S)-\chi (S^{\ast
}).
\]
Summing up, 
\[
\chi (N)=\frac{1}{2}(\chi (S)+\chi (S^{\ast })). 
\]

Since $H^{\ast }(N)\cong\check{H}^{\ast }(K)$ we have that $\chi(N)=\chi(K)$ and, using this fact in the formula from Corollary~\ref{cor:genPoinH} we get that

\[
I(F_{|_N})=\frac{1}{2}(\chi (S^*)-\chi (S))=-\chi
(N)+\chi (S^{\ast }). 
\]

Now suppose that $n=2$ and $K$ is an arbitrary isolated invariant
continuum. We only have to see that the equality $I(F_{|_N})=\chi (K)$ ensures the non-saddleness of $K$ since the converse statement is just case (1). By (\cite[Theorem~10]{BSJMA}), $K$ is non-saddle if and only if
all the components of $S$ are circles .The equality $I(F_{|_N})=\chi (K)$ implies that $%
\chi (S)=0$ and, thus, in this case, no component of $S$ can be a (possibly degenerate) topological interval
and, consequently, must be a circle. Therefore, $K$ is non-saddle. 
\end{proof}

In the following result, we use the Alexandrov (or one-point)
compactification of the Euclidean space $\mathbb{R}^{n}\cup \{\infty \}$ to
show that under further assumptions more can be said about the degree of $F$
and its index. We recall that $\mathbb{R}^{n}\cup \{\infty \}$ is homeomorphic to the $n-$sphere $S^{n}$.

\begin{proposition}
Suppose that $K$ is a non-saddle continuum, $n$ is even and every component of $(\mathbb{R}^{n}\cup \{\infty \})\setminus K$ is contractible
(this always happens for $n=2$). Consider a connected isolating block $N$
of $K.$ Then $\deg (F,N)\leq 1$, and $I(F_{|_N})\leq 1$ when defined. Also, if $
\deg(F,N)=1$ then $K$ is either an attractor or a repeller.
\end{proposition}

\begin{proof}
Let $k$ be the number of components of $(\mathbb{R}^n\cup\{\infty\})\setminus K$ (which is finite and coincides with $1+\rk \check{H}^{n-1}(K)$ by Alexander duality). Since
they are contractible, they have Euler characteristic equal to one. By
Alexander duality%
\begin{align*}
\chi (K)=\chi (\mathbb{R}^{n}\cup\{\infty\})-\chi ((\mathbb{R}^{n}\cup{\infty})\setminus K)\\
=2-k\leq 1. 
\end{align*}
Then, Proposition~\ref{prop:ns1} ensures that $\deg(F,N)=\chi (K)\leq 1$. Furthermore, the equality $%
2-k=1$ holds if and only if $K$ does not disconnect $\mathbb{R}^{n}\cup\{\infty\}$. In
such a case, the only component of $\mathbb{R}^{n}\cup \{\infty \}\setminus K$ is
locally attracted or locally repelled by $K$ by \cite[Theorem~25]{BSdis}. In the first case $K$ is an
attractor and in the second $K$ is a repeller.
\end{proof}

Next we analyze the situation when the dimension $n$ is odd and greater than one.

\begin{proposition}
Suppose that $K$ is a non-saddle continuum, $n>1$ is odd and 
every component of $(\mathbb{R}^{n}\cup \{\infty \})\setminus K$ is contractible.
Consider a connected isolating block $N$ of $K.$ Then $\deg (F,N)\leq k$,
and $I(F_{|_N})\leq k$ when defined. Furthermore, if $I(F_{|_N})$ is defined,
\begin{enumerate}
\item $I(F_{|_N})=k$ if and only if $K$ is an attractor.

\item $I(F_{|_N})=-k$ if and only if $K$ is a repeller

\item $I(F_{|_N})=0$ if and only if $K$ decomposes $\mathbb{R}^{n}$ in an even
number of components, half of them locally attracted by $K$ and half of them
locally repelled.
\end{enumerate}
Analogous statements hold for $\deg (F,N)$. In this case there is no
requirement that $I(F_{|_N})$ be defined.
\end{proposition}

\begin{proof}
By \cite[Theorem~25]{BSdis} we have that all the components of $(\mathbb{R}^{n}\cup\{\infty\})\setminus K$ are either locally attracted or locally repelled by $K$. We denote by $U$ the union of all the components of $(\mathbb{R}^{n}\cup\{\infty\})\setminus K$ which are locally repelled by $K$ and by $V$ the union of all the
components which are locally attracted. Then there is an attractor $A\subset
U$ such that $U$ is the basin of attraction of $A$ and analogously, a
repeller $R\subset V$ such that $V$ is the basin of repulsion of $R$. Let $k$
be the number of components of $\mathbb{R}^{n}\setminus K$ and $u$ the number of
components of $U$. Then, by \cite[Theorem~3.6]{KapRod} $H^{\ast }(U)=\check{H}^{\ast }(A)$ and, by Alexander
duality, $\check{H}^{\ast }(A)=H_{n-\ast }(U,U\setminus A)$. So, since $n$ is odd, we have%
\[
\chi (A)=-\chi (U,U\setminus A)=-\chi (U)+\chi (U\setminus A)=-u+\chi (S). 
\]%
For the last equality, we use the facts that all the components of $U$ are
contractible and that $S$ is a strong deformation retract of $U\setminus A$. Since $%
\chi (A)=\chi (U)=u$ we obtain that $\chi (S)=2u.$

By an analogous argument applied to $R$ and $V$ we obtain that $\chi
(S^{\ast })=2(k-u).$ So, by Proposition~\ref{prop:ns1} we get
\[
I(F_{|_N})=\frac{1}{2}(\chi (S^{\ast }))-\chi (S)=\frac{1}{2}
(2(k-u))-2u)=k-2u, 
\]%
and, since $u\geq 0$, we obtain that $I(F_{|_N})\leq k$.

Also, $I(F_{|_N})=k$ if and only if $u=0,$ which happens if and only if $K$
is an attractor. On the other hand, the equality $I
(F_{|_N})=-k$ holds if and only if $u=k,$ which occurs if and only if $K$ is a
repeller. Finally $I(F_{|_N})=0$ if and only if $k=2u$,
i.e, if and only if the number of components of $\mathbb{R}^{n}-K$ that are
locally repelled by $K$ matches the number of components that are locally
attracted by $K$. 
\end{proof}

\section{Brouwer degree and connecting orbits in attractor repeller decompositions}

In this section we show how to use the Brouwer degree to detect the existence of connecting orbits in attractor-repeller decompositions. We also give an estimate of the Euler characteristics of the set of connecting orbits. Finally, we present an application of these results in order to calculate the Brouwer degree of the Lorenz vector field in an isolating block of the
Lorenz strange set.

We recall that if $K$ is an isolated invariant set and $A\subsetneq K$ is an attractor for the restriction flow $\varphi_{|_K}$, then the set
\[
R=\{x\in K\mid \omega(x)\cap A=\emptyset\}
\]
is non-empty and is a repeller for $\varphi_{|_K}$. The pair $\{A,R\}$ is called \emph{attractor-repeller decomposition} of $K$. Notice that if $K\neq A\cup R$ the orbit of any point $x\notin A\cup R$ satisfies that $\omega(x)\subset A$ and $\omega^*(x)\subset R$. These kind of orbits are the so-called \emph{connecting orbits} between $A$ and $R$.

\begin{proposition}\label{prop:connect}
Let $\{A,R\}$ be an attractor-repeller decomposition of the isolated
invariant set $K$ and $N$ an isolating block of $K$. If $\deg(F,N)\neq
\chi (A)+\chi (R)-\chi (S)$ then there exists an orbit in $K$
connecting $A$ and $R$. Moreover, if $C$ is the union of all connecting
orbits then 
\[
\chi (C)=\chi (A)+\chi (R)-\chi (K).
\]
\end{proposition}

\begin{proof}
We argue by contradiction. Suppose that  there is no orbit in $K$ connecting $A$ and $R$.
Then $K$ is the disjoiunt union of $A$ and $R$. As a consequence, 
\[
\chi (K)=\chi(A)+\chi (R).
\]
Thus, Corollary~\ref{cor:genPoinH} ensures that
\[
\deg(F,N)=\chi
(A)+\chi (R)-\chi (S)
\] 
contradicting the hypothesis. 

Let us compute the Euler characteristic of the set $C$ of connecting orbits. Since $C$ is parallelizable we can find a section  $C_{0}$ of $C$. Define $K_{1}=A\cup C_{0}\lbrack 0,\infty )$ and 
$K_{2}=R\cup C_0(-\infty ,0].$ Then $K=K_{1}\cup K_{2}$ and $K_{1}\cap
K_{2}=C_{0}$. By using the Mayer-Vietoris sequence 
\[
\cdots\longrightarrow \check{H}^{q}(K_{1}\cup K_{2})\longrightarrow \check{H}^{q}(K_{1})\oplus
\check{H}^{q}(K_{2})\longrightarrow \check{H}^{q}(K_{1}\cap K_{2})\longrightarrow \cdots 
\]%
we readily get that 
\[
\chi (K)=\chi (K_{1})+\chi (K_{2})-\chi (C_{0}).
\]
However, $C_{0}$ is a strong deformation retract of $C$ and so $
\chi (C)=\chi (C_{0}).$ Moreover, arguing in the same way as in the proof of Theorem~\ref{thm:1} we obtain that $%
\Sh(K_{1})=\Sh(A)$ and $\Sh(K_{2})=\Sh(R)$ and therefore $\chi
(K_{1})=\chi (A)$ and $\chi (K_{2})=\chi (R)$. Hence, the result follows.
\end{proof}

 The Lorenz vector field $F:\mathbb{R}^{3}\longrightarrow 
\mathbb{R}^{3}$ provides a simplified model of fluid convection dynamics in
the atmosphere, and is given by 
\[
F(x,y,z)=(\sigma (y-x),rx-y-xz,xy-bz),
\]
where $\sigma ,r$ and $b$ are three real positive parameters corresponding
respectively to the Prandtl number, the Rayleigh number and an adimensional
magnitude. We consider the so-called classical values $\sigma=10,b=8/3.$ In \cite{Spa} it is
shown that for values of $r$ between $13.926\ldots$ (which corresponds to the
homoclinic bifurcation) and $24.06$ (where another type of bifurcation
occurs involving the two branches of the unstable manifold of the origin) a
\textquotedblleft strange set\textquotedblright\ $\mathcal{L}$ originates that
exhibits sensitive dependence on initial conditions. For these values of the
parameter $r$, the global attractor $\Omega $ of the Lorenz system has an
attractor-repeller decomposition $\{A,\mathcal{L}\}$ where $\mathcal{L}$ is the Lorenz strange
set and $A$ consists of two points. It follows from \cite[Corollary~5.3]{GSIAM} and \cite[Theorem~7]{Sanjhopf} that the Lorenz strange set has the shape of a wedge of two circles. Therefore, $\chi(\mathcal{L})=-1$.  We shall also make use of the fact that, since $\Omega $ is the global attractor, 
$\chi (\Omega )=1$. 

\begin{proposition}
Let $F:\mathbb{R}^{3}\longrightarrow \mathbb{R}^{3}$ be the Lorenz vector field
and let $N$ be an isolating block of the Lorenz strange set $L.$ Then $%
\deg(F,N)=1.$ Moreover, if $\hat{N}$ is an isolating block of the
global attractor $\Omega$, then $\deg(F,\hat{N})=-1$.
\end{proposition}

\begin{proof}
Let $C$ be the set of connecting orbits between $A$ and $\mathcal{L}$. Then, Proposition~\ref{prop:connect} together with the considerations made before the statement of the proposition ensure that 
\[
\chi (C)=\chi (A)+\chi (\mathcal{L})-\chi (\Omega )=0. 
\]%
Since $\Omega $ is an attractor, the unstable manifold of $\mathcal{L}$ is contained
in $\Omega $ and, thus, agrees with $C.$ Now, let $N$ be an isolating block
of $\mathcal{L}$. By Corollary~\ref{cor:genPoinH} we get 
\[
\deg(F,N)=(-1)^{3}(\chi (\mathcal{L})-\chi (S))=(-1)^{3}(\chi
(\mathcal{L})-\chi (C))=1. 
\]

This contrasts with the situation for the global attractor: if $\hat{N}$ is
an isolating block of $\Omega $ then $\deg(F,\hat{N})=(-1)^{3}\chi ($ $%
\Omega )=-1$.
\end{proof}

\begin{remark}
For values of the parameter $r>24.06$ the strange set $\mathcal{L}$ becomes an
attractor (the Lorenz attractor) and its \v Cech cohomology (even its shape) remains
that of a wedge of two circles. Since $\mathcal{L}$ is now an attractor, $S=\emptyset $ and 
\[
\deg(F,N)=(-1)^{3}\chi (\mathcal{L})=1.
\]
\end{remark}

\section{A generalitation of Borsuk's and Hirsch's antipodal theorems}

We now present a result that is a form of Borsuk's \cite[Theorem~5.2, pg. 163]{OuRu} and Hirsch's \cite[Theorem~5.3, pg. 166]{OuRu} antipodal
theorems for domains which are isolating blocks, involving the dynamics of
the flow induced by $F$ rather than the Brouwer degree of $F$. Using this
result it is possible to conclude from inspection of $K$ and $S$ the
existence of a point $x$ in the boundary $\partial N$ such that the vector
field $F$ points in the same (or opposite) direction at $x$ and $-x$. 

We say that an isolating block $N\subset\mathbb{R}^n$ is \emph{symmetric} if $x\in N$ if and only if $-x\in N$, i.e., $N$ is invariant for the antipodal action. 
\begin{proposition}
Suppose that the isolating block $N$ is symmetric and $0\in N$. Then
\begin{enumerate}
\item[(i)] If $\chi (K)$ and $\chi (S)$ have the same parity then there
is some $x\in \partial N$ such that $F(x)$ and $F(-x)$ point in the same
direction. In particular, if $N$ is the unit ball $B^{n}$ (and thus $\partial
N=S^{n-1})$ and $F_{|_{S^{n-1}}}$ maps $S^{n-1}$ into $S^{n-1}$ then there is some 
$x\in S^{n-1}$ such that $F(x)=F(-x)$.

\item[(ii)] If $\chi (K)$ and $\chi (S)$ have different parity then there
is some $x\in \partial N$ such that $F(x)$ and $F(-x)$ point in opposite
directions. In particular, if $N$ is the unit ball $B^{n}$ (and thus $\partial
N=S^{n-1})$ and $F_{|_{S^{n-1}}}$ maps $S^{n-1}$ into $S^{n-1}$ then there is some 
$x\in S^{n-1}$ such that $F(x)=-F(-x)$.
\end{enumerate}
\end{proposition}

\begin{proof}
If $\chi (K)$ and $\chi (S)$ have the same parity then by Theorem~\ref{cor:genPoinH} the degree of $F_{|_{\mathring{N}}}$ is even and (i) is a consequence of the Hirsch
theorem. If $\chi (K)$ and $\chi (S)$ have a different parity then
by Corollary~\ref{cor:genPoinH} the degree of $F_{|_{\mathring{N}}}$ is odd and (i) is a consequence of the
Borsuk antipodal theorem.
\end{proof}

\bibliographystyle{plain}
\bibliography{Biblio1}

\end{document}